\documentclass[10pt,twoside]{amsart}
\usepackage[all]{xy}
        \usepackage {amssymb,latexsym,amsthm,amsmath,mathtools,dsfont}
        \usepackage{enumitem,color}

\usepackage{array}
\usepackage{float}

\allowdisplaybreaks[4]
\usepackage{hyperref}
\usepackage[capitalize,nameinlink]{cleveref} 
\usepackage{xcolor} 
\hypersetup{
    colorlinks,
    linkcolor={red!80!black},
    citecolor={green!80!black},
    urlcolor={blue!80!black}
}
\usepackage[maxbibnames=99]{biblatex} 
\usepackage{mathtools}

\long\def\symbolfootnote[#1]#2{\begingroup%
\def\thefootnote{\fnsymbol{footnote}}\footnote[#1]{#2}\endgroup}
\newcommand{\PSL}{\mathrm{PSL}}
\newcommand{\SL}{\mathrm{SL}}
\newcommand{\irr}{\mathrm{Irr}}
\newcommand{\fq}{\mathbb{F}}

\makeatletter
\def\imod#1{\allowbreak\mkern10mu({\operator@font mod}\,\,#1)}
\makeatother
\makeatletter
\renewcommand*\env@matrix[1][*\c@MaxMatrixCols c]{%
  \hskip -\arraycolsep
  \let\@ifnextchar\new@ifnextchar
  \array{#1}}
\makeatother

\usepackage{makecell, caption, chngcntr} %

\usepackage{siunitx} %
\counterwithin{table}{section}

\newtheorem{theorem}{Theorem}[section]
\newtheorem{lemma}[theorem]{Lemma}

\newtheorem*{theorem*}{Theorem}
\theoremstyle{definition}

\numberwithin{equation}{section}

\newcommand{\ignore}[1]{}

\newcommand{\mynote}[1]{}
\begin{document}
\setcounter{section}{0}
\title{Character covering number of $\mathrm{PSL}_2 (q)$}
\author[Arvind N.]{Namrata Arvind}
\email{namchey@gmail.com}
\address{The Institute of Mathematical Sciences, 4th Cross St, CIT Campus, Tharamani, Chennai, Tamil Nadu 600113, India}
\author[Panja S.]{Saikat Panja}
\email{panjasaikat300@gmail.com}
\address{Harish-Chandra Research Institute- Main Building, Chhatnag Road, Jhusi, Uttar Pradesh 211019, India}
\thanks{The first named author is partially supported by the IMSc postdoctoral fellowship and the second author has been partially supported by the HRI postdoctoral fellowship.}
\date{\today}
\begin{abstract}
    For a group $G$ and a character $\chi$ of $G$, let 
    $c(\chi)$ denote the set of all irreducible characters of $G$, occurring in $\chi$. We prove that whenever $q\geq 8$, all non-trivial 
    irreducible character $\chi$ of $\PSL_2(q)$ satisfies 
    $c(\chi^4)=\irr\left(\PSL_2(q)\right)$ if $q=2^{2m+1}$ and
    $c(\chi^3)=\irr\left(\PSL_2(q)\right)$ otherwise.
\end{abstract}
\maketitle
\section{Introduction}\label{sec:intro}
Let $G$ be a finite group, $\irr(G)$ be the set of its irreducible characters and for any character $\chi$ of $G$ let us denote the set of all irreducible constituent of
it by $c(\chi)$. First proved in \cite[Ch. XV, Th. IV]{Burn55} by Burnside and later 
refined by Steinberg \cite{Ste62} and Brauer \cite{Br64} is a theorem which 
states that for any faithful representation $\chi$, there exists a positive integer $t(\chi)>0$ such that
\begin{align*}
    c(\chi)\cup c(\chi^2)\cup \cdots\cup c(\chi^{t(\chi)})=\irr(G).
\end{align*}
Furthermore when the center of $\chi\in\irr(G)$, i.e. $Z(\chi)=\{g\in G:|\chi(g)|=\chi(1)\}$ is trivial, Miller has shown that (see \cite[Lemma 4]{AMiller20}) there exists smallest $e(\chi)$ such that 
\begin{align*}
    c(\chi^{e(\chi)})=\irr(G).
\end{align*}
The numbers $\max\limits_{\substack{\chi\in\irr(G)\\\chi\text{ faithful}}}e(\chi)$ will be referred to be \emph{character covering number}. 
This terminology is motivated by the covering number problem in the case of conjugacy 
classes (for example see \cite{ElGoNi99}, \cite{LieSi23} and the references therein). There is a slightly stronger definition of
character covering numbers for groups due to Arad, Chillag, and Herzog, which calls $n$ to be a covering number of a group
$G$, if for all $\chi\in\irr(G)\setminus\{1\}$, once has
$c(\chi^n)=\irr(G)$ (see \cite{AraChHe86}). They further prove that
a finite nontrivial group $G$ has a finite character covering number if and only if $G$ is a nonabelian simple group (see \cite[Theorem 1]{AraChHe86}).
Our definition of character covering number is weaker, as
we are varying characters over the set of irreducible faithful characters.
Miller proved that (see \cite[Theorem 1]{AMiller20}) for $n>4$,
and $X=\{\chi\in\irr(S_n):\chi(1)>1\}$
\begin{enumerate}
    \item $c(\chi^k)=\irr(S_n)$ for all $\chi\in X$ if and only if $k\geq n-1$,
    \item $c(\chi)\cup c(\chi^2)\cup\cdots\cup c(\chi^k)=\irr(S_n)$ for all
    $\chi\in X$ if and only if $k\geq n-1$.
\end{enumerate}
We prove a similar result for $\PSL_2(q)$. Our main result is the following.
\begin{theorem}\label{thm:main}
    Let $q\geq 8$ be a power of prime and $G(q)=\PSL_2(q)$ be the projective special linear group (which is also simple). Then the covering number is $4$ if $q=2^{2m+1}$ for some $m$ and $3$ otherwise.
\end{theorem}
In \cref{sec:prelim} we mention the conjugacy class and
character table of $\PSL_2(q)$. In \cref{sec:proof-of-theorem} we will prove \cref{thm:main}.
\subsection*{Acknowledgement} A part of this work has been carried out during the second-named author's visit to IMSc in January of 2024. He would like to thank the
institute for its support during his stay.
\section{Some background results}\label{sec:prelim}
Let $\tau$ be a generator of $\fq_{q^2}^\times$ and fix $\sigma=\tau^{q+1}$, $\tau_0=\tau^{q-1}$. Then define
\begin{align*}
    S(a)  = \begin{pmatrix}
        \sigma^a & \\ & \sigma^{-a}
    \end{pmatrix} \text{and, }
    T(b)  = \begin{pmatrix}
      ~ & 1\\ 1& \tau_0^{b}+\tau_0^{bq}
    \end{pmatrix}.
\end{align*}.
Also if $\eta$ is a non-square in $\fq_q$, set $N'=\begin{pmatrix}
    1 & \eta \\ &1
\end{pmatrix}$ and set $N=\begin{pmatrix}
    1 & 1 \\ & 1
\end{pmatrix}$
\subsection{Conjugacy classes and character table}
First, we look at the case when $q$ is a power of $2$.
Then the conjugacy classes of $\PSL_2(q)=\SL_2(q)$ have representatives as
\begin{align*}
    I=\begin{pmatrix}
        1 & \\ & 1
    \end{pmatrix},
      N=  \begin{pmatrix}
        1 & 1\\ & 1
    \end{pmatrix},
         S(a), T(b),
\end{align*}
where $1\leq a\leq \frac{q}{2}-1$, $1\leq b\leq \frac{q}{2}$. Since there are $q+1$ many conjugacy classes, $\PSL_2(q)$ has $q+1$ irreducible representations up to equivalence. We have the representations 
$\psi_1$ and $\psi_q$ of dimensions $1$ and $q$ respectively. For each 
$1\leq k\leq \frac{q}{2}-1$ we have representations $\psi^{(k)}_{q+1}$ of dimension
$q+1$ and for each $1\leq j\leq \frac{q}{2}$, we have representations
$\psi^{(j)}_{q-1}$ of dimension $q-1$. The description of these representations 
can be found in \cite[pp. 104-105]{GeMa20}. We present the character table in \cref{table-psl(2)}. Note that here $\epsilon \in \mathbb{C}$ is a primitive $q-1$ root of unity and $\eta_0 \in \mathbb{C}$ is a primitive $q+1$ root of unity.
\begin{table}[H]
\setcellgapes{5pt}\makegapedcells
  \centering%
  \begin{tabular}{*{5}{c}}
    \hline
     &    &     & $1\leq a\leq \frac{q}{2}-1$ & $1\leq b\leq \frac{q}{2}$\\
    $x$& $I$ & $N$ & $S(a)$ & $T(b)$ \\
     $|x^G|$ & $1$ & $q^2-1$ & $q(q+1)$ & $q(q-1)$\\
     \hline
    $\psi_1$ & $1$ & $1$ & $1$ & $1$ \\
    $\psi_q$ & $q$ & $\cdot$ & $1$ & $-1$ \\
    $\psi_{q+1}^{(k)}$ & $q+1$ & $1$ & $\epsilon^{ak}+\epsilon^{-ak}$ &  $\cdot$\\
    $\psi_{q-1}^{(j)}$ & $q-1$ & $-1$ &  $\cdot$ & $-\eta_0^{bj}-\eta_0^{-bj}$ \\
    \hline
  \end{tabular}
    \caption{Character table of $\PSL_2(2^m)$}\label{table-psl(2)}
\end{table}
Similarly when $q$ is odd, the group $\PSL_2(q)$ has $\frac{q+5}{2}$ conjugacy classes. Hence $\PSL_2(q)$ has $\frac{q+5}{2}$ irreducible representations upto equivalence. These representations are obtained by looking at non-faithful irreducible representations of $\SL_2(q)$.

In \cref{table-psl(q)34} we present the character table of $\PSL_2(q)$ in case $q\equiv 3\pmod{4}$, where we have $k=2,4,\cdots, \frac{q-3}{2}$, $j=2,4,\cdots, \frac{q-3}{2}$, $\omega=\frac{1+\sqrt{-q}}{2}$ and $\omega^*=\frac{1-\sqrt{-q}}{2}$.
\begin{table}[H]
  \setcellgapes{5pt}\makegapedcells
  \centering%
  \begin{tabular}{*{7}{c}}
    \hline
    & & & & $1\leq a\leq \dfrac{q-3}{4}$ & $1\leq b\leq\dfrac{q-3}{4}$ &\\
    $x$& $I$ & $N$ & $N'$ & \makecell{$S(a)$} & \makecell{$T(b)$} & $T\left(\frac{q+1}{4}\right)$ \\
    $|x^G|$& $1$ & $\dfrac{q^2-1}{2}$ & $\dfrac{q^2-1}{2}$ & $q(q+1)$ & ${q(q-1)}$ & $\dfrac{q(q-1)}{2}$ \\
     \hline
    $\psi_1$ & $1$ & $1$ & $1$ & $1$ & $1$ & $1$ \\
    $\psi_q$ & $q$ & $\cdot$ & $\cdot$ & $1$ & $-1$ & $-1$ \\
    $\psi_{q+1}^{(k)}$ & $q+1$ & $1$ & $1$ & $\epsilon^{ak}+\epsilon^{-ak}$ & $\cdot$ & $\cdot$\\
    $\psi_{q-1}^{(j)}$ & $q-1$ & $-1$ & $-1$ & $\cdot$ & $-\eta_0^{bj}-\eta_0^{-bj}$ & $-2\eta_0^{\frac{(q+1)j}{4}}$ \\
    $\psi_-'$ & $\frac{q-1}{2}$ & $-\omega^*$ & $-\omega$ & $\cdot$ & $(-1)^{b+1}$ & $(-1)^{\frac{q+5}{4}}$ \\
    $\psi_-''$ & $\frac{q-1}{2}$ & $-\omega$ & $-\omega^*$ & $\cdot$ & $(-1)^{b+1}$ & $(-1)^{\frac{q+5}{4}}$ \\
    \hline
  \end{tabular}
      \caption{Character table of $\PSL_2(q)$, $q\equiv3\pmod{4}$}\label{table-psl(q)34}
\end{table}
\begin{table}[H]
  \setcellgapes{5pt}\makegapedcells
  \centering%
  \begin{tabular}{*{7}{c}}
    \hline
    &&&& $1\leq a\leq \dfrac{q-5}{4}$ && $1\leq b\leq \dfrac{q-1}{4}$\\
    $x$& $I$ & $N$ & $N'$ & \makecell{$S(a)$} & $S\left(\frac{q-1}{4}\right)$ & \makecell{$T(b)$} \\
    $|x^G|$& $1$ & $\dfrac{q^2-1}{2}$ & $\dfrac{q^2-1}{2}$ & $q(q+1)$ & $\dfrac{q(q+1)}{2}$ & $q(q-1)$ \\
     \hline
    $\psi_1$ & $1$ & $1$ & $1$ & $1$ & $1$ & $1$ \\
    $\psi_q$ & $q$ & $\cdot$ & $\cdot$ & $1$ & $1$ & $-1$ \\
    $\psi_{q+1}^{(k)}$ & $q+1$ & $1$ & $1$ & $\epsilon^{ak}+\epsilon^{-ak}$ & $2\epsilon^{\frac{(q-1)k}{4}}$ & $\cdot$\\
        $\psi_{q-1}^{(j)}$ & $q-1$ & $-1$ & $-1$ & $\cdot$ & $\cdot$ & $-\eta_0^{bj}-\eta_0^{-bj}$ \\
    $\psi_+'$ & $\frac{q+1}{2}$ & $\omega$ & $\omega^*$ & $(-1)^a$ & $(-1)^{\frac{q-1}{4}}$ & $\cdot$ \\
    $\psi_+''$ & $\frac{q+1}{2}$ & $\omega^*$ & $\omega$ & $(-1)^a$ & $(-1)^{\frac{q-1}{4}}$ & $\cdot$ \\
    \hline
  \end{tabular}
  \caption{Character table of $\PSL_2(q)$, $q\equiv1\pmod{4}$}\label{table-psl(q)14}
\end{table}
In \cref{table-psl(q)14} we present the character table of $\PSL_2(q)$ in case $q\equiv 1\pmod{4}$, where we have $k=2,4,\cdots, \frac{q-5}{2}$, $j=2,4,\cdots, \frac{q-1}{2}$, $\omega=\frac{1+\sqrt{q}}{2}$ and $\omega^*=\frac{1-\sqrt{q}}{2}$.

\section{Proof of \cref{thm:main}}\label{sec:proof-of-theorem}
\subsection{When $q$ is even}\label{sec:q-even}
We start with the following lemma, which evaluates sums of roots of unity.
\begin{lemma}\label{lem:sum-roots-evn}
    Let $q$ be power of $2$, $1\leq k\leq q/2-1$, $1\leq j\leq q/2$. Further $\epsilon$ and $\eta_0$ be $q-1$-th and $q+1$-th primitive roots of unity respectively. Then for any integer $t>0$,
    \begin{equation}
        \sum\limits_{a=1}^{q/2-1}(\epsilon^{tk})^{a}+(\epsilon^{tk})^{-a}=\begin{cases}
            -1&\text{ if }(tk,q-1)\neq (q-1)\\
            q-2&\text{ otherwise}    
        \end{cases},
    \end{equation} and
    \begin{equation}
        \sum\limits_{b=1}^{q/2}(\eta_0^{tj})^{b}+(\epsilon^{tj})^{-b}=\begin{cases}
            -1&\text{ if }(tj,q+1)\neq (q+1)\\
            q&\text{ otherwise}    
        \end{cases}.
    \end{equation}
\end{lemma}
\begin{proof}
    Note that $\{1-q/2,2-q/2,\cdots,-1,0,1,2,\cdots,q/2-1\}$ is a complete set of residue modulo $q-1$. Hence if $(tk,q-1)\neq (q-1)$, we get
    \begin{align*}
        \sum\limits_{a=1}^{q/2-1}(\epsilon^{tk})^{a}+(\epsilon^{tk})^{-a}=-1+\sum\limits_{a=0}^{q-2}(\epsilon^{tk})^{a}.
    \end{align*}
    Since $\sum\limits_{a=0}^{q-2}(\epsilon^{tk})^{a}$ can be grouped into groups of $(tk,q-1)$ terms each having a zero-sum, the first conclusion follows.
    If $(tk,q-1)=q-1$, then the sum is $q-2$, as there are $q-2$ many terms. The second equality can be proved easily.
    
\end{proof}
Direct calculations prove the following equalities;
$$\begin{array}{cc}
    \langle \psi_q^2,\psi_1 \rangle=1,& \langle \psi_q^2, \psi_q\rangle=1\end{array}$$
    and using \cref{lem:sum-roots-evn}
\begin{align*}
    (q^3-q)\left\langle \left(\psi_{q+1}^{(k)}\right)^2,\psi_1 \right\rangle&=(q+1)^2+(q^2-1)+q(q+1)\sum\limits_{a=1}^{q/2-1}\left(\epsilon^{ai}+\epsilon^{-ai}\right)^2\\
    &=2q^2+2q+q(q+1)\left\{\sum\limits_{a=1}^{q/2-1}\left(\epsilon^{2ai}+\epsilon^{-2ai}\right)+(q-2)\right\}\\
    &=q^3-q,
\end{align*}
which implies $\left\langle \left(\psi_{q+1}^{(k)}\right)^2,\psi_1 \right\rangle=1$. Using similar computation we have $\left\langle \left(\psi_{q+1}^{(k)}\right)^2,\psi_q \right\rangle=2$. 


Note that
\begin{equation}\label{eq:psi-q-psi-k-2}
(q^3-q)\langle \psi_q^2,\psi_{q+1}^{(k)}\rangle
= q^2(q+1)+0+q(q+1)\sum\limits_{a=1}^{q/2-1}\left(\epsilon^{ak}+\epsilon^{-ak}\right)+0.
\end{equation}
\textcolor{black}{Since $k<q-1$, we have $(k,q-1)\neq q-1$. Hence by \cref{lem:sum-roots-evn} the \cref{eq:psi-q-psi-k-2} reduces to
\begin{align*}
    q^2(q+1)-q(q+1)=q^3-q,
\end{align*}}
which in turn proves that $\color{blue}\langle\psi_q^2,\psi_{q+1}^{(k)}\rangle=1$. Next, we have
\begin{align*}
    (q^3-q)\left\langle \psi_q^2,\psi_{q-1}^{(j)}\right\rangle=q^2(q-1)-q(q+1)\sum\limits_{b=1}^{q/2}\left(\eta_0^{bj}+\eta_0^{-bj}\right).
\end{align*}
Since $\eta_0$ is a primitive $q+1$-th root of unity, arguing as before 
we get that
$\color{blue}\langle \psi_q^2,\psi_{q-1}^{(j)}\rangle=1$. This proves that $\color{blue}c(\psi_q^2)=\irr(G(q))$. Similar calculations further 
show that $\color{blue}\left\langle \left(\psi^{(k)}_{q+1}\right)^2,\psi_q\right\rangle=2,$
Now,
\begin{align*}
    &(q^3-q)\left\langle \left(\psi^{(k)}_{q+1}\right)^2,\psi_{q+1}^{(k')}\right\rangle\\
    =&(q+1)^2+(q^2-1)+q(q+1)\sum\limits_{a=1}^{q/2-1}(\epsilon^{ak}+\epsilon^{-ak})^2(\epsilon^{ak'}+\epsilon^{-ak'})\\
    =&\begin{cases}
        q^3-q&\text{when }2k+k'\neq q-1\\
        2(q^3-q)&\text{otherwise}
    \end{cases}.
\end{align*}
Since for a given $k$, there exists a unique $k'$ such that $(2k+k',q-1)=q-1$, we have that
$\left\langle \left(\psi^{(k)}_{q+1}\right)^2,\psi_{q+1}^{(k')}\right\rangle=1$ for $q/2-2$ choices of $k'$ and for the rest it takes the value $2$ . This proves that $\color{blue}c\left(\left(\psi^{(k)}_{q+1}\right)^2\right)=\irr(G(q)).$
Although $\left\langle\left(\psi_{q-1}^{(j)}\right)^2,\psi_1\right\rangle=1$, we have 
\begin{align*}
    \left\langle\left(\psi^{(j)}_{q-1}\right)^2,\psi_q\right\rangle=q(q+1)^2-q(q^2-q)+q(q-1)=0.
\end{align*}
Hence we do not have $c\left(\left(\psi_{q-1}^{(j)}\right)^2\right)=\irr(\PSL_2(q))$. Also, when $3|q+1$,
$c\left(\left(\psi_{q-1}^{((q+1)/3)}\right)^3\right)\neq\irr(\PSL_2(q))$,
since $\left\langle\left(\psi_{q-1}^{((q+1)/3)}\right)^3,\psi_1\right\rangle=0$. However, we have the following values,
$$
\begin{array}{cc}
  \left\langle\left(\psi_{q-1}^{(j)}\right)^3,\psi_1\right\rangle=\begin{cases}
      1&\text{when }3j\neq q+1\\
      0&\text{when }3j=q+1
  \end{cases},   &  
  \left\langle\left(\psi_{q-1}^{(j)}\right)^3,\psi_q\right\rangle=\begin{cases}
      q-3&\text{when }3j\neq q+1\\
      q-2 &\text{when }3j=q+1
  \end{cases},
\end{array}
 $$
 \begin{align*}
     \left\langle\left(\psi_{q-1}^{(j)}\right)^3,\psi_{q-1}^{(j')}\right\rangle=\begin{cases}
      q-4&\text{when }(3j+j',q+1)\neq q+1\\
      q-3&\text{when }(3j+j',q+1)=q+1
      \end{cases},
 \end{align*}
     and
$ \left\langle\left(\psi_{q-1}^{(j)}\right)^3,\psi_{q+1}^{(k)}\right\rangle=q-2,$ which proves that
if $3\not|(q+1)$, then $c\left(\left(\psi_{q-1}^{(j)}\right)^3\right)=\irr(\PSL_2(q))$. In case $3|(q+1)$, we have the
following multiplicities for $j=(q+1)/3$.
$$\begin{array}{c|ccccc}
 \hline
 \chi& \psi_1 & \psi_q & \psi_{(q+1)}^{(k)} & \psi_{q-1}^{(j)} & \psi_{q-1}^{(j')}\\
 \hline
 \left\langle\left(\psi_{q-1}^{(j)}\right)^4,\chi\right\rangle&q-1 & q^2-4q+4 & q^2-3q+3 & q^2-5q+6 & q^2-5q+11\\
 \hline
\end{array}.$$
This proves \cref{thm:main} in case $q$ is even.

\subsection{When $q$ is odd}\label{sec:q-odd} This will be divided into two parts as the character table is dependent on the 
parity of $q$, modulo $4$. We start with the case $q\equiv 3 \pmod{4}$
and later derive the results for $q\equiv 1\pmod{4}$.
\subsubsection{Case I: $q\equiv 3\pmod{4}$} 
Let us start with the following lemma, analogous to \cref{lem:sum-roots-evn}.
\begin{lemma}\label{lem:sum-roots-odd-3}
    Let $q\equiv 3\pmod{4}$ be power of an odd prime, $k$ and $j$ be even integers satisfying $2\leq k\leq (q-3)/2$, $2\leq j\leq (q-3)/2$. Further $\epsilon$ and $\eta_0$ be $q-1$-th and $q+1$-th primitive roots of unity respectively. Then for any integer $t>0$,
    \begin{equation}
        \sum\limits_{a=1}^{(q-3)/4}(\epsilon^{tk})^{a}+(\epsilon^{tk})^{-a}=\begin{cases}
            -1&\text{ if }(tk,q-1)\neq (q-1)\\
            \frac{q-3}{2}&\text{ otherwise}    
        \end{cases},
    \end{equation} and
    \begin{equation}
        \sum\limits_{b=1}^{(q-3)/4}\left\{(\eta_0^{tj})^{b}+(\epsilon^{tj})^{-b}\right\}=\begin{cases}
            -1-\eta_0^{\frac{q+1}{4}\cdot tj}&\text{ if }\left({tj},{q+1}\right)\neq (q+1)\\
            \frac{q+1}{2}&\text{ otherwise}    
        \end{cases}.
    \end{equation}
\end{lemma}
\begin{proof}
    We note that $\left\{-\frac{q-3}{4},-\frac{q-3}{4}+1,\cdots,-1,0,1,\cdots,\frac{q-3}{4}\right\}$ is a complete set of residue modulo $q-1$. Also 
    $\left\{-\frac{q-3}{4},-\frac{q-3}{4}+1,\cdots,-1,0,1,\cdots,\frac{q-3}{4},\frac{q+1}{4}\right\}$ is a complete set of residue modulo $q+1$. The rest of the proof follows similarly, as explained in \cref{lem:sum-roots-evn}.
\end{proof}
We first have the following equalities;
$$
\begin{array}{cccc}
     \left\langle\psi_q^2,\psi_1\right\rangle=1,& \left\langle\psi_q^2,\psi_q\right\rangle=2, & \left\langle\psi_q^2,\psi_-'\right\rangle=1,  &
     \left\langle\psi_q^2,\psi_-''\right\rangle=1.
\end{array}
$$
Note that 
\begin{align*}\label{eq:psi-q-psi-q+1-k-3}
\frac{q^3-q}
{2}\left\langle\psi_q^2,\psi_{q+1}^{(k)}\right\rangle&=q^2(q+1)+q(q+1)\sum\limits_{a=1
}^{(q-3)/4}(\epsilon^{ak}+\epsilon^{-ak})\\
&=q^2(q+1)+q(q+1)\cdot(-1),   
\end{align*}
the last equality follows from \cref{lem:sum-roots-odd-3}. Hence
$\left\langle \langle\psi_q^2,\psi_{q+1}^{(k)}\right\rangle=2$.
Using similar argument we get that $\left\langle\psi_q^2,\psi_{q-1}^{(j)}\right\rangle=2$. We present the all values of the inner product of an irreducible 
character with the square of a non-trivial character 
in \cref{table-psl(q)34}. Before that, we present two computations, 
firstly the case of $\left\langle\left(\psi_{q+1}^{(k)}\right)^2,\psi_{q+1}^{(k')}\right\rangle$ and $\left\langle\left(\psi_{q-1}^{(j)}\right)^2,\psi'_-\right\rangle$. For any $k$, there exists
a unique $k'$ such that $2k+k'=q+1$ or $2k-k'=q+1$. We first deal with this case.
We have
\begin{align*} \frac{q^3-q}{2}\cdot\left\langle\left(\psi_{q+1}^{(k)}\right)^2,\psi_{q+1}^{(k')}\right\rangle&=(q+1)^3+\frac{q^2-1}{2}\cdot1+\frac{q^2-1}{2}\cdot1\\
&+q(q+1)\cdot\left\{\sum\limits_{a=1}^{(q-3)/4}\left(\epsilon^{2ak}+2+\epsilon^{-2ak}\right)\left(\epsilon^{ak'}+\epsilon^{-ak'}\right)\right\}\\
&=(q+1)^3+(q^2-1)+q(q+1)\left\{\sum\limits_{a=1}^{(q-3)/4}\left(\epsilon^{(2k+k')a}+\epsilon^{-(2k+k')a}\right)\right.\\
+&\left.2\sum\limits_{a=1}^{(q-3)/4}\left(\epsilon^{ak}+\epsilon^{-ak'}\right)+\sum\limits_{a=1}^{(q-3)/4}\left(\epsilon^{(2k-k')a}+\epsilon^{-(2k-k')a}\right)\right\}\\
&=(q+1)^3+(q^2-1)+q(q+1)\left(\dfrac{q-3}{2}-3\right)\\
&=\dfrac{3(q^3-q)}{2}.
\end{align*}
Now assume $k'$ to be such that none of $2k+k',2k-k'$ is equal to 
$q+1$. Then we get that
\begin{align*}
    \frac{q^3-q}{2}\cdot\left\langle\left(\psi_{q+1}^{(k)}\right)^2,\psi_{q+1}^{(k')}\right\rangle&=(q+1)^3+(q^2-1)+q(q+1)\left(-4\right)\\
    &=q^3-q.
\end{align*}
This implies that
\begin{align*}
\left\langle\left(\psi_{q+1}^{(k)}\right)^2,\psi_{q+1}^{(k')}\right\rangle=\begin{cases}
      3 & \text{if }2k+k'=q+1\text{ or }2k-k'=q+1\\
      2 & \text{otherwise}
    \end{cases}.
\end{align*}
For the second computation, We have
\begin{align*}
        \left\langle \left(\psi_{q-1}^{(j)}\right)^2,\psi_-'\right\rangle
        =& \frac{(q-1)^3}{2}-\frac{q^2-1}{2}\cdot(\omega+\omega^*)\\
        +&
        q(q-1)\sum\limits_{b=1}^{(q-3)/4}(\eta_0^{bj}+\eta_0^{-bj})^2(-1)^{b+1}\\+&2q(q-1)\eta_0^{(q+1)j/2}(-1)^{(q+5)/4}.
\end{align*}
Then 
\begin{align*}
    &q(q-1)\sum\limits_{b=1}^{(q-3)/4}(\eta_0^{bj}+\eta_0^{-bj})^2(-1)^{b+1}+2q(q-1)\eta_0^{(q+1)j/2}(-1)^{(q+5)/4}\\
    =&q(q-1)\left\{\sum\limits_{b=1}^{(q-3)/4}(\eta_0^{2bj}+\eta_{0}^{-2bj}+2)(-1)^{b+1}+2(-1)^{(q+5)/4}\right\}\\
    =&q(q-1)\left\{-\sum\limits_{b=1}^{(q-3)/4}((-\eta_0^{2j})^b+(-\eta_0^{2j})^{-b})+2\sum\limits_{b=1}^{(q+1)/4}(-1)^{b+1}\right\}.
\end{align*}
Since $\eta_0$ is a primitive $q+1$-th root of unity, $j$ is even and $q\equiv 3\pmod{4}$ (and hence $(q+1)/2$ is even), we get that
$\eta_0^{2j}$ is a $(q+1)/4$-th primitive root of unity. Now we divide this into two cases. 
\begin{enumerate}
    \item First assume that
$(q+1)/4$ is even. Then $-\eta_0^{2j}$ is again a $(q+1)/4$-th primitive root of unity. Hence $\sum\limits_{b=1}^{(q-3)/4}(-\eta_0^{2j})^b=-1$ and $\sum\limits_{b=1}^{(q+1)/4}(-1)^{b+1}=0$. Thus, in this case, we get
\begin{align*}
        \left\langle \left(\psi_{q-1}^{(j)}\right)^2,\psi_-'\right\rangle
        =2.
\end{align*}
\item Next let us assume $(q+1)/4$ is odd. In this case
$-\eta_0^{2j}$ is a primitive $(q+1)/2$-th root of unity. Now,
\begin{align*}
    \sum\limits_{b=1}^{(q-3)/4}((-\eta_0^{2j})^b+(-\eta_0^{2j})^{-b})=\sum\limits_{b=1}^{(q+1)/4}((-\eta_0^{2j})^b+(-\eta_0^{2j})^{-b}),
\end{align*}
since for $b=(q+1)/4$, we have $(-\eta_0^{2j})^b+(-\eta_0^{2j})^{-b}=0$. Thus, in this case as well we have that
\begin{align*}
        \left\langle \left(\psi_{q-1}^{(j)}\right)^2,\psi_-'\right\rangle
        =2.
\end{align*}
\end{enumerate}

Using similar calculations we get the table \cref{table:inner-product-3}, where $(a,b)$-th entry is given by the inner product of the $a$-th element of the first column and the $b$-th element of the first row. In this table $k'$ is the positive number such that
either $2k+k'$ or $2k-k'$ is equal to $q-1$. Also, $j'$ is the positive number such that
either $2j+j'$ or $2j-j'$ is equal to $q+1$. The $k''\neq k'$ are the integers ranging over other values of $k$, and similarly $j''\neq j'$.
\begin{table}[H]
    \centering
    \begin{tabular}{c|cccccccc}
        \hline
        $\chi$ & $\psi_1$ & $\psi_q$ & $\psi_{q+1}^{(k')}$ & $\psi_{q+1}^{(k'')}$ & $\psi_{q-1}^{(j')}$ & $\psi_{q-1}^{(j'')}$ & $\psi'_{-}$ & $\psi''_{-}$\\
        \hline
$\left\langle\left(\psi_{q+1}^{(k)}\right)^2,\chi\right\rangle$
& $1$ & $3$ & $3$ & $2$ & $2$ & $2$ & $1$ & $1$\\
$\left\langle\left(\psi_{q-1}^{(j)}\right)^2,\chi\right\rangle$
& $1$ & $1$ & $2$ & $2$ & $1$ & $2$ & $1$ & $1$\\
$\left\langle\left(\psi'_{-}\right)^3,\chi\right\rangle$
& $1$ & $\frac{q-3}{4}$ & $\frac{q+1}{4}$ & $\frac{q+1}{4}$ & $\frac{q-7}{4}$ & $\frac{q-7}{4}$ & $\frac{q-3}{4}$ & $\frac{q-3}{4}$\\
$\left\langle\left(\psi''_{-}\right)^3,\chi\right\rangle$
& $1$ & $\frac{q-3}{4}$ & $\frac{q+1}{4}$ & $\frac{q+1}{4}$ & $\frac{q-7}{4}$ & $\frac{q-7}{4}$ & $\frac{q-3}{4}$ & $\frac{q-3}{4}$\\
\hline
    \end{tabular}
    \caption{Inner product of the characters}
    \label{table:inner-product-3}
\end{table}

This finishes the proof of our claim. Hence in this case the covering number is $3$.

\subsubsection{Case II: $q\equiv 1\pmod{4}$} 
Before we look at the interaction of powers of a character of $\PSL_2(q)$ with its other characters we state a lemma similar to the previous cases.
\begin{lemma}\label{lem:sum-roots-odd-1}
    Let $q\equiv 1\pmod{4}$ be power of an odd prime, $k$ and $j$ be even integers satisfying $2\leq k\leq (q-5)/2$, $2\leq j\leq (q-1)/2$. Further $\epsilon$ and $\eta_0$ be $q-1$-th and $q+1$-th primitive roots of unity respectively. Then for any integer $t>0$,
    \begin{equation}
        \sum\limits_{b=1}^{(q-1)/4}(\eta^{tk})^{a}+(\eta^{tk})^{-a}=\begin{cases}
            -1&\text{ if }(tk,q-1)\neq (q+1)\\
            \frac{q-1}{2}&\text{ otherwise}    
        \end{cases},
    \end{equation} and
    \begin{equation}
        \sum\limits_{a=1}^{(q-5)/4}\left\{(\epsilon_0^{tj})^{b}+(\epsilon^{tj})^{-b}\right\}+2\epsilon_0^{\frac{q-1}{4}tj\cdot}=\begin{cases}
            -1+(-1)^{tj/2}&\text{ if }\left(\frac{tj}{2},\frac{q+1}{2}\right)\neq (q-1)/2\\
            \frac{q+1}{2}&\text{ otherwise}    
        \end{cases}.
    \end{equation}
\end{lemma}
\begin{proof}
    Similar to proof of \cref{lem:sum-roots-odd-3}
\end{proof}
It is easy to check the following equalities;
$$
\begin{array}{cccc}
     \left\langle\psi_q^2,\psi_1\right\rangle=1,& \left\langle\psi_q^2,\psi_q\right\rangle=2, & \left\langle\psi_q^2,\psi_+'\right\rangle=1,  &
     \left\langle\psi_q^2,\psi_+''\right\rangle=1.
\end{array}
$$
Now 
\begin{equation}\label{eq:psi_q square psi_q+1 k 1}
    \frac{q^3-q}
{2}\left\langle \psi_q^2,\psi_{q+1}^{(k)}\right\rangle
= q^2(q+1) +q(q+1)\sum\limits_{a=1
}^{(q-5)/4}(\epsilon^{ak}+\epsilon^{-ak}) +q(q+1)\epsilon^{\frac{q-1}{4}k}.
\end{equation}
Using \cref{lem:sum-roots-odd-1} the sum in \cref{eq:psi_q square psi_q+1 k 1} reduces to  $\left\langle\psi_q^2,\psi_{q+1}^{(k)}\right\rangle=2$
 for all $k$.
Similarly, we observe that 
$\left\langle\psi_q^2,\psi_{q-1}^{(j)}\right\rangle=2$
for all $j$ that occurs as characters of $\PSL_2(q)$.
Since $\left\langle{\psi'_+}^2,\psi''_+\right\rangle=0$, we look at the inner product of ${\psi'_+}^3$ with all characters of $\PSL_2(q)$ to get the following :
$$
\begin{array}{cccc}
     \left\langle{\psi'_+}^3,\psi_1\right\rangle=1,& \left\langle{\psi'_+}^3,\psi_q\right\rangle=\frac{q+3}{4}, & \left\langle{\psi'_+}^3,\psi_{q+1}^k\right\rangle=\frac{q+7}{4},  &\\
      \left\langle{\psi'_+}^3,\psi_{q_1}^j\right\rangle=\frac{q-1}{4}, &
     \left\langle{\psi'_+}^3,\psi'_+\right\rangle=\frac{q+7}{4},  &
     \left\langle{\psi'_+}^3,\psi''_+\right\rangle=1.&
\end{array}
$$
Note that the interaction of ${\psi''_+}^3$ with the irreducible representations of $\PSL_2(q)$ will also be similar to that of ${\psi'_+}^3$. 

Thus it is remaining to check that some power of the characters $\psi_{q+1}^k$ and  $\psi_{q-1}^j$ contains all irreducible representations of $\PSL_2(q)$. 
Using \cref{lem:sum-roots-odd-1} we get the following:
$$
\begin{array}{ccccc}
     \left\langle{\psi_{q+1}^k}^2,\psi_1\right\rangle=1,& \left\langle{\psi_{q+1}^k}^2,\psi_q\right\rangle=3, & \left\langle{\psi_{q+1}^k}^2,\psi_{q-1}^j\right\rangle=2,  &\\
     \left\langle{\psi_{q+1}^k}^2,\psi'_+\right\rangle=1,  &
     \left\langle{\psi_{q+1}^k}^2,\psi''_+\right\rangle=1.&
\end{array}
$$
Now similar to before we see that for each $k$, there is a unique $k'$ such that $2k+k'=q-1$ or $2k-k'= q-1$. In this case we notice that $\left\langle \langle{\psi_{q+1}^(k)}^2,\psi_{q+1}^{(k')}\right\rangle=3$. In the case when $k'$ is such that neither $2k+k'$ nor $2k-k'$ is equal to $q-1$ we get that  $\left\langle \langle{\psi_{q+1}^(k)}^2,\psi_{q+1}^{(k')}\right\rangle=2$. 

Hence the character covering number for $\psi_{q+1}^{(k)}$ is equal to 3. Similarly it can be shown that ${\psi_{q-1}^{(j)}}^2$ contains all irreducible representations of $\PSL_2(q)$.


The table \cref{table:inner-product-1} summarises all the character values, where $(a,b)$-th entry is given by the inner product of the $a$-th element of the first column and the $b$-th element of the first row. Note that here $k''$ denotes the unique number such that $2k+k''=q-1$ or $2k-k''=q-1$ and $j''$ denotes the unique number such that $2j+j''=q+1$ or $2j-j''=q+1$.

\begin{table}[H]
   \centering
  \begin{tabular}{ccccccccc}
     \hline
    $\chi$&$\psi_1$ & $\psi_q$ & $\psi_{q+1}^{(k')}$ & 
    $\psi_{q+1}^{(k'')}$ &
    $\psi_{q-1}^{(j')}$ & 
    $\psi_{q-1}^{(j'')}$ &
               $\psi_+'$ & $\psi_+''$\\
       \hline
      $\left\langle\left(\psi^{(k)}_{q+1}\right)^2,\chi\right\rangle$ & $1$ & $3$ & $2$ & $3$& $2$ & $2$& $1$ & $1$ \\
        $\left\langle\left(\psi^{(j)}_{q-1}\right)^2,\chi\right\rangle$ & $1$ & $1$ & $2$ & $2$& $2$& $1$ & $1$ & $1$\\
     $\left\langle\left(\psi'_+\right)^3,\chi\right\rangle$ & $1$ & $\frac{q+3}{4}$ & $\frac{q+7}{4}$ & $\frac{q+7}{4}$ & $\frac{q-1}{4}$ & $\frac{q-1}{4}$ &  $\frac{q+7}{4}$ & $1$\\
     $\left\langle\left(\psi''_+\right)^3,\chi\right\rangle$ & $1$ & $\frac{q+3}{4}$ & $\frac{q+7}{4}$ & $\frac{q+7}{4}$ & $\frac{q-1}{4}$ &  $\frac{q-1}{4}$ & $\frac{q+7}{4}$ & $1$\\
       
        \hline
    \end{tabular}
    \caption{Inner product of the characters}
    \label{table:inner-product-1}
\end{table}

       

This completes the proof of \cref{thm:main}.
\printbibliography

@book {GeMa20,
    AUTHOR = {Geck, Meinolf and Malle, Gunter},
     TITLE = {The character theory of finite groups of {L}ie type},
    SERIES = {Cambridge Studies in Advanced Mathematics},
    VOLUME = {187},
      NOTE = {A guided tour},
 PUBLISHER = {Cambridge University Press, Cambridge},
      YEAR = {2020},
     PAGES = {ix+394},
      ISBN = {978-1-108-48962-1},
   MRCLASS = {20C33 (20D06 20G05)},
  MRNUMBER = {4211779},
MRREVIEWER = {Donald\ L.\ White},
}

@article { AMiller20,
    AUTHOR = {Alexander R. Miller},
     TITLE = {Covering numbers for characters of symmetric groups},
    SERIES = {Ann. Sc. Norm. Sup. Pisa Cl. Sci},
    VOLUME = {TBD},
      NOTE = {},
 PUBLISHER = {},
      YEAR = {2024},
     PAGES = {TBD}, 
}

@book {Burn55,
    AUTHOR = {Burnside, W.},
     TITLE = {Theory of groups of finite order},
      NOTE = {2d ed},
 PUBLISHER = {Dover Publications, Inc., New York},
      YEAR = {1955},
     PAGES = {xxiv+512},
   MRCLASS = {20.0X},
  MRNUMBER = {69818},
}

@article {Br64,
    AUTHOR = {Brauer, Richard},
     TITLE = {A note on theorems of {B}urnside and {B}lichfeldt},
   JOURNAL = {Proc. Amer. Math. Soc.},
  FJOURNAL = {Proceedings of the American Mathematical Society},
    VOLUME = {15},
      YEAR = {1964},
     PAGES = {31--34},
      ISSN = {0002-9939,1088-6826},
   MRCLASS = {20.80},
  MRNUMBER = {158004},
MRREVIEWER = {R.\ Steinberg},
       %DOI = {10.2307/2034344},
       %URL = {https://doi.org/10.2307/2034344},
}

@article {Ste62,
    AUTHOR = {Steinberg, Robert},
     TITLE = {Complete sets of representations of algebras},
   JOURNAL = {Proc. Amer. Math. Soc.},
  FJOURNAL = {Proceedings of the American Mathematical Society},
    VOLUME = {13},
      YEAR = {1962},
     PAGES = {746--747},
      ISSN = {0002-9939,1088-6826},
   MRCLASS = {20.80},
  MRNUMBER = {141710},
MRREVIEWER = {H.\ K.\ Farahat},
       %DOI = {10.2307/2034167},
       %URL = {https://doi.org/10.2307/2034167},
}

@article {LieSi23,
    AUTHOR = {Liebeck, Martin W. and Simion, Iulian I.},
     TITLE = {Covering numbers for simple algebraic groups},
   JOURNAL = {Vietnam J. Math.},
  FJOURNAL = {Vietnam Journal of Mathematics},
    VOLUME = {51},
      YEAR = {2023},
    NUMBER = {3},
     PAGES = {605--616},
      ISSN = {2305-221X,2305-2228},
   MRCLASS = {20G07 (20E45)},
  MRNUMBER = {4631622},
       %DOI = {10.1007/s10013-023-00609-3},
       %URL = {https://doi.org/10.1007/s10013-023-00609-3},
}

@article {ElGoNi99,
    AUTHOR = {Ellers, Erich W. and Gordeev, Nikolai and Herzog, Marcel},
     TITLE = {Covering numbers for {C}hevalley groups},
   JOURNAL = {Israel J. Math.},
  FJOURNAL = {Israel Journal of Mathematics},
    VOLUME = {111},
      YEAR = {1999},
     PAGES = {339--372},
      ISSN = {0021-2172,1565-8511},
   MRCLASS = {20G40 (20D06 20D60)},
  MRNUMBER = {1710745},
MRREVIEWER = {Jean\ Michel},
       %DOI = {10.1007/BF02810691},
       %URL = {https://doi.org/10.1007/BF02810691},
}

@article {AraChHe86,
    AUTHOR = {Arad, Zvi and Chillag, David and Herzog, Marcel},
     TITLE = {Powers of characters of finite groups},
   JOURNAL = {J. Algebra},
  FJOURNAL = {Journal of Algebra},
    VOLUME = {103},
      YEAR = {1986},
    NUMBER = {1},
     PAGES = {241--255},
      ISSN = {0021-8693},
   MRCLASS = {20C15},
  MRNUMBER = {860703},
MRREVIEWER = {Roderick\ Gow},
 %      DOI = {10.1016/0021-8693(86)90183-3},
  %     URL = {https://doi.org/10.1016/0021-8693(86)90183-3},
}
\end{document}